\documentclass[14pt]{article}
\usepackage{amsmath}
\usepackage{amssymb}
\usepackage{mathrsfs}
\usepackage{amscd}
\usepackage{amsthm,color}
\usepackage[matrix,frame,import,curve]{xypic}
\usepackage{authblk}
\usepackage{tikz}
\usepackage{dirtytalk}

\let\margin\marginpar
\newcommand\myMargin[1]{\margin{\raggedright\scriptsize #1}}
\renewcommand{\marginpar}[1]{\myMargin{#1}}

\newtheorem{lemma}{Lemma}[section]
\newtheorem{theorem}[lemma]{Theorem}

\newtheorem{prop}[lemma]{Proposition}
\theoremstyle{definition}
\newtheorem{definition}[lemma]{Definition}

\newtheorem{remark}[lemma]{Remark}

\theoremstyle{remark}
\newtheorem*{proof*}{Proof}
\numberwithin{equation}{section}

\def\Hom{{\mathrm{Hom}}}

\def\CC{{\mathbb C}}

\def\cG{{\mathcal{G}}}

\def\fm{{\mathfrak{m}}}

\def\lgg{{\langle\langle}}
\def\rgg{{\rangle\rangle}}
\def\ol{\overline}
\def\ot{{\otimes}}

\def\wh{\widehat}

\def\Lm{{\Lambda}}

\def\p{{\prime}}
\def\pp{{\prime\prime}}
\def\1{{\bf{1}}}

\def\cy{{\mathrm{cyc}}}
\def\Aut{{\mathrm{Aut}}}
\def\id{{\mathrm{Id}}}
\def\der{{\mathrm{Der}}}
\def\cder{{\mathrm{cDer}}}
\def\dder{{\mathrm{\mathbb{D}er}}}

\def\im{{\mathrm{im}}}

\def\rJ{{\mathrm{J}}}
\def\Ad{{\mathrm{Ad}}}

\title{Quasi-homogeneity of superpotentials}
\date{}
\author[1]{Zheng Hua\thanks{huazheng@maths.hku.hk}}
\author[21]{Gui-Song Zhou \thanks{10906045@zju.edu.cn}}
\affil[1]{Department of Mathematics, the University of Hong Kong, Hong Kong SAR, China}
\affil[2]{Department of Mathematics, Ningbo University, Ningbo, China}

\begin{document}
\maketitle
\begin{abstract}
In this article, we study the quasi-homogeneity of a superpotential in a complete free algebra over an algebraic closed field of characteristic zero. We prove that a superpotential with finite dimensional Jacobi algebra is right equivalent to a weighted homogeneous superpotential if and only if the corresponding class in the 0-th Hochschlid homology group of the Jacobi algebra is zero. This result can be viewed as a noncommutative version of the famous theorem of Kyoji Saito on isolated hypersurface singularities.
\end{abstract}

\section{Introduction}
\newtheorem{maintheorem}{\bf{Theorem}}
\renewcommand{\themaintheorem}{\Alph{maintheorem}}
\newtheorem{mainconjecture}[maintheorem]{\bf{Conjecture}}
\renewcommand{\themainconjecture}{}

This is our second paper studying the Jacobi-finite superpotentials, after \cite{HZ}. Let $F=k\lgg x_1,\ldots,x_n\rgg$ be a complete free algebra over a field $k$. A \emph{superpotential} $\Phi$ refers to an element in the vector space $F_{\ol{\cy}}$ consisting of elements of $F$ modulo cyclic permutations. The cyclic derivative $D_i\Phi$ of $\Phi$ with respect to $x_i$ is an element in $F$. To every $\Phi$, we can associate to it an associative algebra $\Lm(F,\Phi)$, defined to be the quotient of $F$ by the closed two sided ideal generated by $D_i\Phi$ for $i=1,\ldots,n$. We call $\Lm(F,\Phi)$ the \emph{Jacobi algebra} (or superpotential algebra) associated to $F$ and $\Phi$. The Jacobi algebra is an invariant of the superpotential. It is natural to ask to what extent the superpotential is determined by its Jacobi algebra. 

The natural projection from $F$ to $\Lm(F,\Phi)$ induces a natural map from $F_{\ol{\cy}}$ to $\Lm(F,\Phi)_{\ol{\cy}}$. We denote the image of $\Phi$ under this map by $[\Phi]$. In \cite{HZ}, we proved that:
\begin{theorem}(\cite[Theorem B]{HZ})\label{MY}
Assume that $k=\CC$.
Let $\Phi,\Psi\in F_{\ol{\cy}}$ be two superpotentials of order $\geq 3$. Suppose that the superpotential algebras $\Lm(F ,\Phi)$ and $\Lm(F ,\Psi)$ are both finite dimensional.
Then the following two statements are equivalent:
\begin{enumerate}
\item[$(1)$] There is an algebra isomorphism $\gamma : \Lm(F , \Phi)\cong \Lm(F , \Psi)$ so that $\gamma_*([\Phi]) = [\Psi]$. 
\item[$(2)$] $\Phi$ and $\Psi$ are right equivalent.
\end{enumerate}
\end{theorem}
We call a superpotential \emph{quasi-homogeneous} if  $[\Phi]=0$.
It is easy to check that if $\Phi$ is weighted homogeneous (see Definition \ref{def:wh}) then $[\Phi]=0$, i.e. weighted homogeneous $\Rightarrow$ quasi-homogeneous. By the above theorem, weighted homogeneous superpotentials with finite dimensional Jacobi algebras are completely classified by their Jacobi algebras.   The next question is given an arbitrary superpotential, how to determine whether it is right equivalent to a weighted homogeneous superpotential or not? For a superpotential with finite dimensional Jacobi algebra, we show that it is right equivalent to a weighted homogeneous one if and only if it is quasi-homogeneous.

\begin{theorem}(Theorem \ref{Saito})
Assume that $k$ is an algebraic closed field of zero characteristic.
Let $\Phi\in F_{\ol{\cy}}$ be a  superpotential of order $\geq3$ such that the Jacobi algebra associated to $\Phi$ is  finite dimensional. Then $\Phi$ is quasi-homogeneous if and only if  $\Phi$ is right equivalent to a weighted-homogenous superpotential of type $(r_1,\ldots,r_n)$ for some rational numbers $r_1,\ldots,r_n$ lie strictly between $0$ and $1/2$. Moreover, in this case, all such types $(r_1, \ldots, r_n)$ agree with each other up to permutations on the indexes $1,\ldots, n$. 
\end{theorem}

One can make a formal analogue between the study of superpotentials with the study of hypersurface singularities. If we view the complete free algebra as the ring of formal functions on  noncommutative affine space, then Theorem \ref{MY} is a noncommutative version of Mather-Yau theorem (\cite{MY}) and Theorem \ref{Saito} is a noncommutative version of Saito's theorem (\cite{Sait}). In fact, the proofs of Theorem \ref{MY} and \ref{Saito} are to some extent inspired by the proofs of these two classical theorems, although certain conceptional gap needs to be filled in the  noncommutative case.  

Superpotential algebras have appeared in many mathematical areas including representation theory, topology and algebraic geometry. The finite dimensional condition should be understood as an analogue of isolated hypersurface singularity. Finite dimensional superpotential algebras can appear at least from two sourses. The first is the theory of (generalized) cluster category (\cite{Am}). The cluster category is defined from a Ginzburg dg-algebra of dimension 3 with finite dimensional zero-th homology. The zero-th homology of a Ginzburg dg-algebra is a superpotential algebra. It also appears in noncommutative deformation theory. Given a 3-Calabi Yau dg-category with appropriate assumptions, the noncommutative deformation functor of a rigid object in this category is represented by a finite dimensional superpotential algebra. For example, if $C\subset Y$ is a contractible rational curve in a smooth CY 3-fold $Y$ then the corresponding superpotential algebra is precisely the contraction algebra considered by Donovan and Wemyss \cite{DW13}.

It is well known that weighted homogeneous hypersurface singularities admit a lot of good properties. For instance, the monodromy of a weighted homogeneous hypersurface singularity is semi-simple. Weighted homogeneous superpotentials also have some nice properties. For example, the Ginzburg algebra of a weighted homogeneous superpotential carries an extra grading. The calculation of Hochschild cohomology can be greatly simplified using this grading. However, compared with the commutative case the understanding of the properties of weighted homogeneous superpotentials is still quite limited. Theorem \ref{Saito} is one attempt along this line. Note that the vanishing of the class $[\Phi]$ is fairly easy to check. At least at this point, commutative and noncommutative cases have marginal differences. Remember that whether an isolated hypersurface singularity is weighted homogeneous can be checked by comparing the Milnor number and  the Tjurina number. 

The paper is organized as follows. In Section \ref{sec:Pre}, we recall several basic facts on noncommutative calculus and superpotential algebras. These facts are well known to experts and have been reviewed in Section 2 of \cite{HZ}. We repeat it simply to make the paper as self-contained as possible. The readers who are familiar with these can skip Section \ref{sec:Pre}. In Section \ref{sec:Chev}, we prove a Jordan-Chevalley type theorem for decomposition of derivations on complete free algebras. This result is of independent interest. It enables us to link quasi-homogeneous superpotentials to weighted-homogeneous one. In Section \ref{sec:proof}, we present the proof of the main theorem.

\paragraph{Acknowledgments.}  The manuscript was completed during the visit of the first author to University of Washington on August 2018. He would like to thank James Zhang and University of Washington for hospitality.  The research of the first author was supported by RGC General Research Fund no. 17330316,  no. 17308017 and Early Career grant no. 27300214. The research of the second author was supported by NSFC grant no. 11601480 and RGC Early Career grant no. 27300214.

\section{Preliminaries}\label{sec:Pre}

In this section, we collect basic notations and terminologies that are of concern. Throughout, we fix a base commutative ring $k$ with unit. All algebras are $k$-algebras, and we denote $\ot=\otimes_k$ for the tensor product of $k$-modules unless
specified otherwise.

Fix an integer $n\geq1$. Let  $F$ be the complete free algebra $k\lgg x_1,\ldots,x_n\rgg$.  Elements  of $F$ are formal series $\sum_{w} a_w w$, where $w$ runs over all words in $x_1,\ldots,x_n$ and $a_w\in k$. Let $\fm \subseteq F$ be the ideal generated by $x_1,\ldots, x_n$.  For any subspace $U$ of $F$, let $U^{cl}$ be the closure of $U$ with respect to the $\fm$-adic topology on $F$. Note that $U^{cl}= \cap_{r\geq 0 } (U+\fm^r)$.

Recall that \emph{($k$-)derivation of $F$ in a $F$-bimodule  $M$} is defined to be a ($k$-)linear map $\delta:F\to M$  satisfies the Leibniz rule, that is $\delta(ab) = a\delta(b)+a\delta(b)$ for all $a,b\in F$.  We denote by $\der_k(F,M)$ the set of all $k$-derivations of $F$ in $M$, which carries a natural $k$-module structure.  
We write 
\[
\der_k ( F): =\der_k(F,F) 
\]
an call its elements    \emph{$k$-derivations of $F$}. Clearly, derivations  of  $F$  are  uniquely determined by their value at generators $x_j$.  Note that $\der_k(F)$ admits neither left nor right $F$-module structure.

Let $F\wh{\ot} F$ be the $k$-module whose elements are formal series of the form
$\sum\nolimits_{u,v} a_{u,v}~ u\ot v$, where $u,v$ runs over all words in $x_1,\ldots, x_n$ and $a_{u,v}\in k$. This is nothing but the adic completion of $F\ot F$ with respect to the ideal $\fm\ot F+F\ot \fm$. It contains $F\ot F$  as a subspace  under the identification
\[(\sum_{u} a'_u~u) \ot (\sum_{v} a''_v~v) \mapsto \sum_{u,v} a'_ua''_v~ u\ot v.\]
There are two obvious $F$-bimodule structures on $F \wh{\ot} F$, which we call the outer and the inner bimodule structures respectively,  extends those  on the subspace $F \ot F$ defined respectively by 
\[
a(b^\p\ot b^{\p\p})c:=ab^\p\ot b^{\p\p}c \quad \text{and} \quad a*(b^\p\ot b^{\p\p})*c:=b^\p c\ot ab^{\p\p}.
\]
Unless otherwise stated, we  view $F\wh{\ot} F$ as a  $F$-bimodule with respect to the outer bimodule structure.

We call derivations of $F$ in the $F$-bimodule $F\wh{\ot}F$ \emph{double derivations} of $F$. The inner bimodule structure on $F\wh{\ot}F$ naturally yields a bimodule structure on the space of double derivations
\[
\dder_k(F):= \der_k(F, F\wh{\ot}F).
\]
For any  $\delta \in \dder_k(F)$ and any $f\in F$, we also write $\delta(f)$ in Sweedler's notation as \begin{align}\label{Sweedler}
\delta(f)= \delta(f)^\p\ot\delta(f)^\pp.
\end{align}
One shall bear in mind  that this notation is an infinite sum. Clearly, double derivations of $F$ are uniquely determined by their values on generators $x_j$. Thus, we have double derivations
\[\frac{\partial~}{\partial x_i}: F \to F\wh{\ot} F, ~~~~ x_j\mapsto \delta_{i,j}~1\ot 1.\]
Moreover, every double derivation of $F$ has a unique representation of the form
\begin{align}\label{representation-doub}
\sum_{i=1}^n \sum_{u,v} a_{u,v}^{(i)}~ u* \frac{\partial~}{\partial x_i} * v,~~~~~a_{u,v}^{(i)}\in k,
\end{align}
where $u,v$ run over all words on $x_1,\ldots, x_n$, and $*$ denotes the scalar multiplication of the bimodule structure of $\dder_k(F)$. The infinite sum (\ref{representation-doub}) makes sense in the obvious way. 


There are two obvious  linear maps $\mu: F\wh{\ot} F\to F$ and $\tau: F\wh{\ot} F \to  F\wh{\ot} F$ given respectively by
\[
\mu(\sum_{u,v} a_{u,v} u\ot v) =  \sum_{w} (\sum_{w=uv} a_{u,v})~w \quad \text{and} \quad \tau (\sum_{u,v} a_{u,v} u\ot v) = \sum_{u,v} a_{v,u} u\ot v.
\]
Also, putting  on $\Hom_k(F,F)$  the $F$-bimodule structure  defined by
\[
a_1\cdot f \cdot a_2: b\mapsto a_1 f(b) a_2, ~~~~~ f\in \Hom_k(F,F), ~ a_1, a_2, b\in F.
\]
Though  the map $\dder_k(F) \xrightarrow{\mu\circ- } \Hom_k(F,F)$ doesn't preserves  bimodule structures,  the map
\[
\mu\circ \tau\circ- : \dder_k (F) \to \Hom_k(F,F)
\]
is clearly a homomorphism of $F$-bimodules. We write
$$
\cder_k(F):= \im(\mu \circ \tau \circ-)
$$
and call its elements \emph{cyclic derivations} of $F$. 
Note that by definition $\cder_k(F)$ is  an $F$-sub-bimodule of $\Hom_k(F,F)$, and hence is itself an $F$-bimodule. For each $1\leq i\leq n$, let
\[
D_{x_i}:= \mu \circ \tau \circ \frac{\partial~}{\partial x_i} \in \cder_k(F).
\]
These cyclic derivations was first studied by Rota, Sagan and Stein \cite{RRS}. By (\ref{representation-doub}), every cyclic derivation of $F$ has a decomposition (not necessary unique) of the form
\begin{align}\label{representation-cyc}
\sum_{i=1}^n \sum_{u,v} a_{u,v}^{(i)}~ u\cdot D_{x_i}\cdot v,~~~~~a_{u,v}^{(i)}\in k.
\end{align}
In the sequel, if there is no risk of confusion, we always simply write $D_i$ for $D_{x_i}$.

Elements of $F_{\ol{\cy}}:= F/[F,F]^{cl}$ are called \emph{superpotentials} of $F$.  Let $\pi: F\to F_{\ol{\cy}}$ be the canonical projection. Given a superpotential $\Phi\in F_{\ol{\cy}}$, there are two linear maps  
\begin{eqnarray*}
&~\, \Phi_{\#}:\der_k(F) \to F_{\ol{\cy}}, & \quad  \xi \mapsto \pi(\xi(\phi))\\
&  \Phi_*: \cder_k(F) \to F,& \quad  D\mapsto D(\phi),
\end{eqnarray*}
where $\phi$ is any representative of $\Phi$. Note that all derivations and cyclic derivations of $F$ are continuous with respect to the $\fm$-adic topology on $F$. Consequently,  $\xi([F,F]^{cl}) \subseteq [F,F]^{cl}$ for each  derivation $\xi\in \der_k(F)$, and $D([F,F]^{cl})=0$ for each cyclic derivation $D\in \cder_k(F)$. It follows immediately that the resulting maps $\Phi_{\#}$ and $\Phi_*$ are  independent of the choise of $\phi$.

\begin{lemma} \label{derivation+}
For any superpotential $\Phi\in F_{\ol{\cy}}$, there is a commutative diagram as following:
\begin{align*}
\xymatrix{
\dder_k (F) \ar@{->>}[r]^-{\mu\circ \tau\circ-}  \ar@{->>}[d]^-{\mu\circ-} &   \cder_k (F) \ar[r]^-{\Phi_*} & F\ar@{->>}[d]^-{\pi} \\
\der_k (F) \ar[rr]^-{\Phi_{\#}} & & F_{\ol{\cy}}.
}  
\end{align*}
Moreover, $\Phi_*$ is a homomorphism of $F$-bimodules  and hence $\im(\Phi_*)$ is a two-sided ideal of $F$.
\end{lemma}

\begin{proof}
Let $\phi\in F$ be an arbitrary representative of $\Phi$. Note that  $\mu(\delta(\phi)) - \mu(\tau (\delta(\phi))) \in [F,F]^{cl}$ for all double derivations $\delta\in \dder_k(F)$ and all formal series $\phi\in F$, the diagram commutes. The surjection of the maps $\pi$, $\mu\circ-$ and $\mu\circ \tau\circ-$ is clear.  Also, we have
\begin{eqnarray*}
\Phi_*(a\cdot D\cdot b) = (a\cdot D\cdot b)(\phi) = a D(\phi) b = a\Phi_*(D) b
\end{eqnarray*}
for all $a,b\in F$ and $D\in \cder_k(F)$, so $\Phi_*$ is a homomorphism of $F$-bimodules.
\end{proof}

Recall that two words $u$ and $v$ on $x_1,\ldots, x_n$ are {\em conjugate} if there are words $w_1,w_2$ such that $u=w_1w_2$ and $v=w_2w_1$. Equivalent classes under this equivalence relation are called \emph{necklaces} or \emph{conjugacy classes}. Also recall that a word $u$ is \emph{lexicographically smaller} than another word $v$ if there exist factorizations $u=wx_iw'$ and $v=wx_jw''$ with $i<j$. This order relation restricts to a total order on  each necklace. Let us call a word \emph{standard}  if it is maximal in its necklace. 

\begin{remark}
Every superpotential of $F$  has a unique representative, called the  \emph{canonical representative},  which is a  formal linear combination of  standard words. Given a superpotential $\Phi\in F_{\ol{\cy}}$, the smallest integer $r$ such that $\Phi\in \pi(\fm^r)$ is called the \emph{order of $\Phi$}.  Note that the order of a superpotential coincides with the order of its canonical representative.
\end{remark}

\begin{definition}\label{superpotential-complete}
Let $\Phi\in F_{\ol{\cy}}$ be a superpotential. The \emph{Jacobi algebra} or the \emph{superpotential algebra} associated to $\Phi$ is defined to be the associative algebra
\[
\Lm(F, \Phi) := F/\rJ(F,\Phi),
\]
where $\rJ(F,\Phi):= \im(\Phi_*)$ is called the \emph{Jacobi ideal} of $F$ associated to $\Phi$. Note that if $k$ is noetherian  then $\rJ(F,\Phi)= (\Phi_*(D_{x_1}),\ldots, \Phi_*(D_{x_n}))^{cl}$ by \cite[Lemma 2.6]{HZ}.
\end{definition}

We denote by $\mathcal{G}:=\Aut_k(F,\fm)$ the group of $k$-algebra automorphisms of $F$ that preserve $\fm$. It is a subgroup of  $\Aut_k(F)$, the group of all $k$-algebra automorphisms of $F$.  In the case when $k$ is a field,  $\mathcal{G}=\Aut_k(F)$. Note that $\mathcal{G}$ acts on $F$ and $F_{\ol{\cy}}$ in the obvious way.

\begin{definition}\label{right-equivalent}
For superpotentials $\Phi,\Psi\in F_{\ol{\cy}}$, we say $\Phi$ is \emph{(formally) right equivalent} to $\Psi$ and write $\Phi \sim \Psi$, if $\Phi$ and $\Psi$ lie in the same $\cG$-orbit.
\end{definition}

\begin{prop}[\text{\cite[Proposition 3.7]{DWZ}, \cite[Proposition 3.3]{HZ}}]\label{Jacobi-transform}
Let $\Phi \in F_{\ol{\cy}}$  and $H\in \mathcal{G}$.Then
\[
H(\rJ(F,\Phi)) =\rJ(F, H(\Psi)).
\]
Consequently, $H$ induces an isomorphism of algebras $\Lm(F, \Phi) \cong \Lm(F, H(\Psi))$.
\end{prop}

Given a superpotential $\Phi\in F_{\ol{\cy}}$,  let $\fm_\Phi: =\fm/J(F,\Phi)$, which is an ideal of $\Lm(F,\Phi)$.  By Lemma \cite[Lemma 2.8]{HZ}, the $\fm_\Phi$-adic topology of $\Lm(F,\Phi)$ is complete.  Let
\[
\Lm(F, \Phi)_{\ol{\cy}}:=\Lm(F, \Phi)/[\Lm(F, \Phi), \Lm(F, \Phi)]^{cl}.
\]
Note that if $\Lm(F,\Phi)$ is finitely generated as a $k$-module then 
\[
\Lm(F, \Phi)_{\ol{\cy}} = \Lm(F, \Phi)/[\Lm(F, \Phi), \Lm(F, \Phi)] = HH_0(\Lm(F,\Phi)).
\]
The projection map $F\to \Lm(F, \Phi)$ induces  a natural map 
\[p_\Phi: F_{\ol{\cy}} \to \Lm(F, \Phi)_{\ol{\cy}} \]
with kernel $\pi(\rJ(F,\Phi))$. For any $\Theta\in F_{\ol{\cy}}$, we  write 
$[\Theta]$ for the class $p_\Phi(\Theta)$ in $\Lm(F, \Phi)_{\ol{\cy}}$.

\begin{definition}
A superpotential $\Phi\in F_{\ol{\cy}}$ is said to be \emph{quasi-homogeneous} if the class $[\Phi]$ is zero in $\Lm(F,\Phi)_{\ol{\cy}}$, or equivalently  $\Phi $ is contained in $\pi(\rJ(F,\Phi))$.
\end{definition}

The following result on quasi-homogeneous superpotentials is of interest. It is an immediate consequence of Theorem \ref{MY}.

\begin{theorem}[\text{\cite[Corollary 3.9]{HZ}}]
Let  $k$ be the complex number field. Let $\Phi, \Psi\in F_{\ol{\cy}}$ be two quasi-homogeneous superpotentials of order $\geq 3$ such that the Jacobi algebras $\Lm(F,\Phi)$ and $\Lm(F,\Psi)$ are both finite dimensional.  Then $\Phi$ is right equivalent to $\Psi$ if and only if $\Lm(F,\Phi)\cong \Lm(F,\Psi)$ as algebras.
\end{theorem}

\begin{definition}\label{def:wh}
Let $(r_1,\ldots, r_n)$ be a tuple of rational numbers with $0<r_1,\ldots, r_n\leq 1/2$. A superpotential $\Phi\in F_{\ol{\cy}}$ is said to be \emph{weighted-homogeneous of type $(r_1,\ldots,r_n)$} if it has a representative which is a  linear combination of monomials $x_{i_1}x_{i_2}\cdots x_{i_p}$ such that $r_{i_1}+r_{i_2}+\ldots r_{i_p}=1$.
\end{definition}

\begin{lemma}\label{weighted-imply-quasi}
Let $\Phi\in F_{\ol{\cy}}$ be a superpotential that is right equivalent to a weighted-homogeneous superpotential of a certain type. Then $\Phi$ is quasi-homogeneous.
\end{lemma}

\begin{proof}
By Proposition \ref{Jacobi-transform}, quasi-homogeneous superpotentials are closed under the action of $\mathcal{G}$.  So we may assume $\Phi$ is itself weighted-homogeneous of type $(r_1,\ldots, r_n)$ for some rational numbers $0< r_1,\ldots, r_n\leq 1/2$.  It is not hard to see that $\Phi=\pi\Big(\sum_{i=1}^nr_ix_i \cdot \Phi_*(D_{x_i})\Big)$. The result follows.
\end{proof}

The aim of this paper is to study  the converse of Lemma \ref{weighted-imply-quasi}.  To this end, we employ the following geometric view of point. Let $\der_k^+(F)$ be the space of derivations of $F$ that send $\mathfrak{m}$ to $\fm$. Intuitively,  $\der_k^+(F)$ can be seen as the ``tangent space'' of the ``infinite dimensional Lie group'' $\mathcal{G}$ at the identity map $\id$. For every superpotential $\Phi\in F_{\ol{\cy}}$, the action of $\mathcal{G}$ on $F_{\ol{\cy}}$ yields a ``smooth'' map
\[
\lambda_\Phi: \mathcal{G} \to F_{\ol{\cy}}, \quad H\mapsto H(\Phi).
\]
The  map $\Phi_{\#}= \Phi_{\#}|_{\der_k^+(F)}: \der_k^+(F)\to F_{\ol{\cy}}$ can be seen as the ``differential'' of  $\lambda_\Phi$ at $\id$.

It is clear that a superpotential $\Phi\in F_{\ol{\cy}}$ is weighted-homogeneous of type $(r_1,\ldots, r_n)$ if and only if $\Phi_{\#}(\xi) =\Phi$, where $\xi\in \der_k^+(F)$ is the derivation given by $\xi(x_i)=r_ix_i$.  We have the following characterization of quasi-homogeneous superpotentials in this perspective.

\begin{lemma}\label{quasi-derivation}
Suppose that $k$ is a field. Let $\Phi\in F_{\ol{\cy}}$ be a superpotential of order $\geq 2$ such that the Jacobi algebra associated to $\Phi$ is finite dimensional. Then $\Phi$ is quasi-homogeneous if and only if $\Phi_{\#}(\xi)=\Phi$ for some derivation $\xi\in \der_k^+(F)$.
\end{lemma}

\begin{proof}
The backward implication is clear by the commutative diagram in Lemma \ref{derivation+}.  Next we show the forward implication. Assume that $\Phi$ is quasi-homogeneous. By \cite[Proposition 3.14 (1)]{HZ},
\[
 \Phi = \pi \Big(\sum_{i=1}^n g_i \cdot \Phi_*(D_{x_i})\Big) = \Big(\pi\circ \Phi_*\Big)(\sum_{i=1}^ng_i\cdot D_{x_i})\]
for some formal series $g_1,\ldots, g_n\in \fm$. Let $\dder_k^+(F)$  be the space of double derivations that map $\fm$ to $\fm\wh{\ot} F +F\wh{\ot} \fm$, and let $\cder_k^+(F)$ be the space of cyclic derivations that map  $\fm$ to $\fm$. Then the commutative diagram in Lemma \ref{derivation+} restricts to a commutative diagram 
\begin{align*}
\xymatrix{
\dder_k^+ (F) \ar@{->>}[r]^-{\mu\circ \tau\circ-}  \ar@{->>}[d]^-{\mu\circ-} &   \cder_k^+ (F) \ar[r]^-{\Phi_*} & F \ar@{->>}[d]^-{\pi} \\
\der_k^+ (F) \ar[rr]^-{\Phi_{\#}} & & F_{\ol{\cy}}. 
}   
\end{align*}
Since $\sum_{i=1}^ng_i\cdot D_{x_i} \in \cder_k^+(F)$,  the above commutative diagram shows that
$$\Phi_{\#}(\xi) = \Big(\pi\circ \Phi_*\Big)(\sum_{i=1}^ng_i\cdot D_{x_i}) = \Phi$$
for some derivation $\xi\in \der_k^+(F)$. This completes the proof.
\end{proof}

\section{Jordan-Chevalley decomposition of derivations}\label{sec:Chev}

Throughout, let $F$ be a fixed complete free algebra $k\lgg x_1,\ldots, x_n\rgg$ over a field $k$, and let $\mathfrak{m}$ be the ideal generated by $x_1,\ldots,x_n$. We assume that $k$ is  algebraically closed. This section devotes to establish a Jordan-Chevalley type decomposition for derivations of $F$ that send $\fm$ to $\fm$. 

The space of all derivations of $F$ that send $\fm$ to $\fm$ is denoted by $\der_k^+(F)$. There is a natural group action of $\mathcal{G}:=\Aut_k(F,\fm)=\Aut_k(F)$ on $\der_k^+(F)$ given by $$\Ad_H\xi:= H\circ \xi \circ H^{-1}.$$ This action respects the Lie bracket on $\der_k^+(F)$. In addition, one has $\xi(f) = b f$ if and only if $(\Ad_H\xi)(H(f)) =b~ H(f)$
for any  $\xi\in \der_k^+(F)$, any $H\in \mathcal{G}$, any $f\in F$ and any $b\in k$.

\begin{definition}
We say that a derivation $\xi\in \der_k^+(F)$
\begin{enumerate}
\item[(1)]  is \emph{nilpotent} if it  induces a nilpotent  endomorphism on $\mathfrak{m}/\mathfrak{m}^2$;
\item[(2)] is \emph{semisimple} if it has $n$ eigenvectors in $\mathfrak{m}$ which form a basis in $\mathfrak{m}/\mathfrak{m}^2$, or equivalently there is an automorphism $H\in\Aut_k (F)$ such that $\Ad_H\xi$ has eigenvectors $x_1,\ldots,x_n$.
\end{enumerate}
\end{definition}

\begin{prop}\label{eigenvector-decom}
Let $\xi\in \der_k^+(F)$ be a semisimple derivation. 
\begin{enumerate}
\item[(1)] A scalar $a\in k$ is an  eigenvalues of $\xi$ if and only if  $a\in \mathbb{N}a_1+\cdots +\mathbb{N}a_n$, where $a_1,\ldots,a_n\in k$ are the eigenvalues of the induced map of $\xi$ on $\mathfrak{m}/\mathfrak{m}^2$.
\item[(2)] Every formal series $f\in F$  can be uniquely decomposed into a formal sum $f=\sum_{a} f_a$, where $a$ runs over eigenvalues of $\xi$ and $f_a$ is an eigenvector of $\xi$ with eigenvalue $a$.
 \end{enumerate}
\end{prop}

\begin{proof}
We may assume $\xi$ has $x_1,\ldots, x_n$ as eigenvectors with eigenvalue $a_1,\ldots, a_n$ respectively. Then every word $w=x_{i_1}\cdots x_{i_p}$ is an eigenvector of $\xi$ with eigenvalue $a_{i_1}+\cdots+a_{i_n}$. The result follows
\end{proof}

\begin{prop}\label{diagonalizable-simultaneous}
Let $\zeta_1,\ldots,\zeta_m\in \der_k^+(F)$ be semisimple derivations that commute with each other, that is $[\zeta_i,\zeta_j]=0$ for all $i,j=1,\ldots,n$. Then there exists an automorphism $H\in \Aut_k(F)$ such that $\Ad_H\zeta_1,\ldots, \Ad_H\zeta_m$ all have $x_1,\ldots, x_n$ as eigenvectors.
\end{prop}

\begin{proof}
We prove it by induction on $m$. For $m=1$ there is nothing to prove. Suppose that the result is true for $m=p$ and we proceed to justify the case that $m=p+1$.  By the induction hypothesis, we may assume in priori that $x_i$ is an eigenvector of $\zeta_j$ with eigenvalue $r_{ij}$ for $i=1,\ldots,n$ and $j=1,\ldots,p$. For any $p$-tuple $a=(a_1,\ldots,a_p)$ of scalars, let $F_{a}$ be the space of formal series which are eigenvectors of $\zeta_j$ with eigenvalue $a_j$ for $j=1\ldots,p$. Since every word is a simutaneously eigenvector of $\zeta_1,\ldots,\zeta_p$,  every formal series $f\in F$  can be uniquely expressed as
$$
f=\sum_{a} f_a, \quad f_a\in F_a,
$$
where $a=(a_1,\ldots,a_p)$ runs over all $p$-tuples of scalars with $a_j$ an eigenvalue of $\zeta_j$ for $j=1,\ldots, p$. Since $\zeta_{p+1}$ commutes with $\zeta_1,\ldots,\zeta_p$, it follows that if $f$ is an eigenvector of $\zeta_{p+1}$ then  $f_a$ is also an eigenvector of $\zeta_{p+1}$ with the same eigenvalue as that of $f$. Indeed, one has $$\zeta_j(\zeta_{p+1}(f_a)) = \zeta_{p+1}(\zeta_j(f_a)) =a\zeta_{p+1}(f_a),\quad   j=1,\ldots, p.$$ 
So if $\zeta_{p+1}(f) = bf$ then $\zeta_{p+1}(f)$ has two decompositions into simultaneous eigenvectors of $\zeta_1,\ldots, \zeta_p$ as
\[
\zeta_{p+1}(f) = \sum_a \zeta_{p+1}(f_a) =\sum_a bf_a.
\]
It follows immediately that $\zeta_{p+1}(f_a) = bf_a$.

Let  $w(x_j) := (r_{j1},\ldots, r_{jp})$ for $j=1,\ldots,n$. Let $X_1,\ldots, X_s$ be the partition of $X=\{x_1,\ldots,x_n\}$ by the relation that $x_i \sim x_j$ if and only if $w(x_i)=w(x_j)$. By permutation, we may assume that 
\[
X_1=\{x_1,\ldots, x_{l_1}\}, ~ X_2=\{x_{l_1+1},\ldots, x_{l_2}\},\ldots, X_s=\{ x_{l_{s-1}+1}, \ldots, x_{n}\}
\]
for some integers $0=l_0< l_1< l_2< \ldots <l_s=n$. Since $\zeta_{p+1}$ is semisimple, it has eigenvectors $f_1,\ldots, f_n\in \mathfrak{m}$  that form a basis of $\mathfrak{m}/\mathfrak{m}^2$.  By the above discussion, the set  
$$Y_i:=\{~ (f_1)_{w(x_{l_i})}, (f_2)_{w(x_{l_i})},\ldots, (f_n)_{w(x_{l_i})} ~\}$$ 
consists of simultaneous eigenvectors of $\zeta_1,\ldots, \zeta_{p+1}$. Moreover, $Y_i$
induces a spanning set of  the subspace $V_i \subseteq \mathfrak{m}/\mathfrak{m}^2$ spanned by $X_i$, so we may choose 
$$h_{l_{i-1}+1},\ldots, h_{l_i} \in Y_i$$ 
which form a basis of $V_i$ for $i=1,\ldots, s$. By the inverse function theorem (cf. \cite[Lemma 2.13]{HZ}), the algebra homomorphism $T: F\to F$ given by $x_i\mapsto h_i$ is an automorphism. We have
$$(\Ad_{T^{-1}}\zeta_{j})(x_i) = T^{-1}(\zeta_j(h_i)) = T^{-1} (r_{ij}h_i) = r_{ij}x_i, \quad i=1,\ldots, n, ~~ j=1,\ldots, p.$$
So $\Ad_{T^{-1}}\zeta_{j} =\zeta_j$ for $j=1,\ldots, p$. In addition, $h_i$ is an eigenvector of $\zeta_{p+1}$ by the construction, so $\Ad_{T^{-1}} \zeta_{p+1}$ has $x_1,\ldots, x_n$ as eigenvectors. Take $H=T^{-1}$, the result follows. 
\end{proof}

A derivation $\xi\in \der_k^+(F)$ is called  \emph{principle} if  $\xi(x_1),\ldots,\xi(x_n)$ are all homogeneous of degree $1$, that is they are all linear combinations of the generators $x_1,\ldots, x_n$; 

\begin{lemma}\label{principle-derivation}
Let $\xi\in \der_k^+(F)$ be a principle derivation with a decomposition $\xi=\xi' +\xi''$ such that $\xi' \in \der_k^+(F)$ is principle semisimple derivation, $\xi''\in \der_k^+(F)$ is principle  nilpotent derivation  and $[\xi',\xi'']=0$. Then for any homogeneous formal series $f\in F$ and any  scalar $b\in k$, there exists a homogeneous formal series $h\in F$ of the same degree as $f$, such that  $$(\xi-b) h - f$$ is  an eigenvector of $\xi'$ with eigenvalue $b$ (eigenvectors always include the zero vector). 
\end{lemma}

\begin{proof}
Suppose $f$ is of degree $p$. Let $F_{(p)}$ be the space of homogeneous formal series of degree $p$. For any scalar $c$,  let $F_{(p; c)}$ be the space of  formal series in $F_{(p)}$ which are eigenvectors of $\xi'$ with eigenvalue $c$. 
From the property that $\xi''$ commutes with $\xi'$,  we have that $\xi''$ acts nilpotently on $F_{(p;c)}$. Since the restriction map of $\xi -b\cdot \id$ on $F_{(p;c)}$ equals to the restriction map of $(c-b)\cdot \id -\xi''$ on $F_{(p;c)}$, it is invertible when $c\neq b$. Note that there exist scalars $c_1,\ldots, c_q\in k$ such that 
$$F_{(p)}= F_{(p;c_1)}\oplus\cdots \oplus F_{(p;c_q)}.$$ 
So $f$ has a decomposition $f=f_1+\cdots +f_q$ with $f_i\in F_{(p;c_i)}$ for $i=1,\ldots,q$. If $c_i\neq b$ then define $h_i \in F_{(p;c_i)}$ to be the preimage of $f_i$ under the restriction map of $\xi'+\xi'' - b\cdot \id$ on $F_{(p;c_i)}$, which is invertible by the above discussion; and if $c_i=b$ then define $h_i=0$. Now consider the formal series $h=h_1+\cdots +h_q \in F_{(p)}$. Clearly, if $b\not \in \{c_1,\ldots,c_q\}$ then $\xi'(h)+\xi''(h) - b\cdot h - f=0\in F_{(p;b)}$; and if $b=c_i$ for some $i$ then $\xi'(h)+\xi''(h) - b\cdot h - f=-f_i\in  F_{(p;b)}$. 
\end{proof}

\begin{theorem}(Jordan-Chevalley decomposition)\label{decom-derivation}
For every derivation $\xi\in \der_k^+(F)$, there exists a unique pair of derivations $\xi_S, \xi_N\in \der_k^+(F)$ such that
\[
\xi = \xi_S+\xi_N, 
\]
$\xi_S$ is semisimple, $\xi_N $ is nilpotent and $[\xi_S,\xi_N]=0$. Moreover, any derivation in $\der_k^+(F)$ commutes with $\xi$ if and only if it commutes with $\xi_S$ and $\xi_N$.
\end{theorem}

The above decomposition of a derivation analogs to the Jordan-Chevalley decomposition of linear endomorphisms of finite dimensional vector spaces over an algebraically closed field (see \cite[Proposition 4.2]{Hump}).
We refer to  $\xi_S$ (resp. $\xi_N$) the \emph{semisimple part} (resp. \emph{nilpotent part}) of $\xi$. 

We will use the following notation in the argument  given below. Let $r\geq0$ be an integer. For any derivation $\eta\in \der_k^+(F)$, we write $\eta_{[r]}$ to be the induced endomorphism of  $\eta$ on  $F/\fm^{r+1}$. Note that if $\eta$ is semisimple (resp. nilpotent) as a derivation then $\eta_{[r]}$ is semisimple (resp. nilpotent) as a linear endomorphism.
For any formal series $f\in F$, we write $f_{(r)}$ (resp. $f_{(\leq r)}$) the sum of terms of degree $r$ (resp. $\leq r$) that occurs in $f$.  In addition, for any  derivation $\eta\in \der_k^+(F)$, we write $\eta_{(r)} \in \der_k^+(F)$ to be the derivation given  by  $x_i\mapsto \eta(x_i)_{(r)}$ for $i=1,\ldots, n$.   

\begin{proof}
First we show the uniqueness of the decomposition.
Suppose that $\xi=\xi_S'+\xi_N'$ and $\xi=\xi_S''+\xi_N''$ are two such decompositions. Then 
\[
\xi_{[s]} =(\xi_S')_{[s]}+(\xi_N')_{[s]} \quad \text{and} \quad \xi_{[s]} =(\xi_S'')_{[s]}+(\xi_N'')_{[s]}.
\]
Since $(\xi_S')_{[s]}, (\xi_S'')_{[s]}$ are semisimple and $(\xi_N')_{[s]},(\xi_N')_{[s]}$ are nilpotent,  one gets $(\xi_S')_{[s]}=(\xi_S'')_{[s]}$ and $(\xi_N')_{[s]} = (\xi_N'')_{[s]}$ by \cite[Proposition 4.2 (a)]{Hump}, for every integer $s\geq0$. Therefore, $\xi_S'= \xi_S''$ and $\xi_N'= \xi_N''$. This prove the uniqueness of the decomposition.

Next we show the last statement.  The converse implication is clear. To see the forward implication, assume $\eta\in \der_k^+(F)$ is a derivation commutes with $\xi$. By \cite[Proposition 4.2 (b)]{Hump}, 
\[
[\eta, \xi_S]_{[s]} = [\eta_{[s]},( \xi_S)_{[s]}] = 0, \quad s\geq0.
\]
Therefore, $[\eta,\xi_S] =0$ and hence $[\eta,\xi_N]= [\eta,\xi] -[\eta,\xi_S]=0$. This prove the last statement.

Finally we show the existence of the decomposition. Note that the action of $\mathcal{G}$ on $\der_k^+(F)$ respects the Lie bracket,  preserves semisimpleness and nilpotentness of derivations.   So we may assume in priori that  the restriction of $\xi_{(1)}$ on $F_{(1)}$ is of the Jordan normal form with respect to the ordered basis $x_1,\ldots,x_n$, that is there exists positive integers $l_1,\ldots,l_r$ with $$l_1+\cdots +l_r=n$$ and scalars $$a_1=\cdots = a_{l_1}, ~
a_{l_1+1}=\cdots = a_{l_1+l_2}, ~ \ldots, ~ a_{l_{1}+\cdots+l_{r-1}+1} = \cdots =a_{n}$$
such that $\xi_{(1)}(x_{i}) = a_i x_i$ for $i=1,l_1+1,l_1+l_2+1,\ldots$, and $\xi_{(1)}(x_i)= a_i x_i+x_{i-1}$ otherwise.  Let  $\xi_{(1)}'$ be the  derivation given by $x_i\mapsto a_i x_i$ for $i=1,\ldots, n$, and let $\xi_{(1)}'':=\xi_{(1)} -\xi_{(1)}''$.  Clearly, $\xi_{(1)}'$ is principle semisimple, $\xi_{(1)}''$ is principle nilpotent, and $[\xi_{(1)}',\xi_{(1)}'']=0$. 

We proceed to recursively construct  an infinite sequence of $n$-tuples of $(h^{(s)}_1,\ldots, h^{(s)}_n)$ of formal series in $F$  for $s\geq1$ such that 
\begin{enumerate}
\item[(1)] $h^{(s)}_i$ is homogeneous of degree $s$  for $i=1,\ldots, n$;
\item[(2)] $\big (\Ad_{ (H^{(s)}\circ \cdots \circ H^{(1)} )} \xi \big )(x_i)_{(\leq s)}$ is an  eigenvector of $\xi_{(1)}'$ with eigenvalue $a_i$ for $i=1,\ldots, n$,  where $H^{(r)}\in \mathcal{G}$ is the automorphism given by $x_j\mapsto x_j+h^{(r)}_j$ for $j=1,\ldots,n$. 
\end{enumerate}
Take $h^{(1)}_1=\cdots = h^{(1)}_n=0$, then the case that $s=1$ is fulfilled. Suppose that the required tuple $(h^{(s)}_1,\ldots, h^{(s)}_n)$ has been constructed for $s=1,\ldots p$. To simplify the notation, let 
$$
\xi^{(p)}:= \Ad_{ (H^{(p)}\circ \cdots \circ H^{(1)} )} \xi.
$$ 
By construction, 
$$
(\xi^{(p)})_{(1)} = \xi_{(1)}
$$
and $\xi^{(p)} (x_i)_{(\leq p)}$ is  an eigenvector of $\xi_{(1)}'$ with eigenvalue $a_i$  for $i=1,\ldots n$. 
By Lemma \ref{principle-derivation}, we may choose a homogeneous formal series $h^{(p+1)}_{i}$  of degree $p+1$ for $i=1, l_1+1,\cdots$  such that
$$
\varphi^{(p+1)}_i:=\big(\xi_{(1)}-a_i \big)(h^{(p+1)}_{i})  -\xi^{(p)}(x_i)_{(p+1)} 
$$
is an eigenvector  of $\xi_{(1)}'$ with eigenvalue $a_i$;  and then apply Lemma \ref{principle-derivation} again,   we may also choose inductively on other $i$ a homogeneous formal series $h^{(p+1)}_{i}$ of degree $p+1$  such that 
$$
\varphi^{(p+1)}_{i}:=\big(\xi_{(1)}-a_i \big)(h^{(p+1)}_{i})  - \big(\xi^{(p)}(x_i)_{(p+1)}+h^{(p+1)}_{i-1} \big)
$$
is  an eigenvector  of $\xi_{(1)}'$ with eigenvalue $a_i$. It is easy to check that
$$(H^{(p+1)})^{-1} : x_i\mapsto x_i - h^{(p+1)}_i + \text{\rm H.O.T.},  \quad i=1,\ldots,n.$$
Here, ${\rm H.O.T.}$ is an abbreviation for ''higher order terms''.
So for $i=1,\ldots,n$ one has
\begin{eqnarray*}
\Big ( \Ad_{ (H^{(p+1)}\circ \cdots \circ H^{(1)} )} \xi \Big)(x_i) 
&=& \big(\Ad_{H^{(p+1)}}\xi^{(p)}\big)(x_i) \\
&=& H^{(p+1)}\big (\xi^{(p)} (x_i-h^{(p+1)}_i) \big) + \text{\rm H.O.T.} \\
&=&\xi^{(p)} \big(x_i-h^{(p+1)}_i \big )_{(\leq p+1)} + H^{(p+1)}\big( \xi^{(p)} (x_i-h^{(p+1)}_i)_{(1)}\big )_{(p+1)} +\text{H.O.T.} \\
&=&\xi^{(p)}(x_i)_{(\leq p+1)} - \xi_{(1)} (h^{(p+1)}_i  ) + H^{(p+1)}\big( \xi_{(1)} (x_i)\big )_{(p+1)} +\text{H.O.T.} \\
&=& \xi^{(p)}(x_i)_{(\leq p)} - \varphi^{(p+1)}_i + \text{\rm H.O.T.}.
\end{eqnarray*}
Here, the third  equality holds because  $H^{(p+1)}(f) = f_{(\leq p+1)} + H^{(p+1)}(f_{(1)})_{(p+1)}$ modulo     
$\fm^{p+2}$ for any formal series $f\in F$; the fourth equality holds because $\xi^{(p)}(h^{(p+1)}_i)_{(\leq p+1)}= \xi_{(1)}(h^{(p+1)}_i)$ and  $\xi^{(p)}(f)_{(1)} =\xi_{(1)}(f_{(1)})$ for any formal series $f\in F$; and the last equality holds  because $\xi_{(1)}(x_i)$ is either $a_ix_i$ or $a_ix_i +x_{i-1}$ depending on $i$. Consequently, $\Big ( \Ad_{ (H^{(p+1)}\circ \cdots \circ H^{(1)} )} \xi \Big)(x_i)_{(\leq p+1)}$  is an eigenvector  of $\xi'_{(1)}$ with eigenvalue $a_i$ for $i=1,\ldots, n$.

Now let $g_i^{(s)}:= \big ( H^{(s)}\circ \cdots \circ H^{(1)} \big) (x_i)$ for $i=1,\ldots,n$ and $s\geq1$. Since
$g^{(s+1)}_{i} -g^{(s)}_i \in \mathfrak{m}^{s+1}$ for $s\geq 1$,  
the infinite sequence $(g_i^{(1)},g_i^{(2)}, g_i^{(3)},\ldots)$ converges to a formal series $g_i$. Clearly,
$$
(g_i)_{(\leq s)}=(g_i^{(s)})_{(\leq s)}, \quad s\geq1.
$$
Let $H\in \mathcal{G}$ be the automorphism given by $H(x_i) = g_i$  for $i=1,\ldots,n$.
It is easy to check that
$$
(\Ad_H\xi)(x_i)_{(\leq s)} = \big (\Ad_{ (H^{(s)}\circ \cdots \circ H^{(1)} )} \xi \big )(x_i)_{(\leq s)}, \quad s\geq1,
$$
so $(\Ad_H\xi)(x_i)$ is an eigenvector of $\xi_{(1)}'$ with eigenvalue $a_i$. In addition, one has 
$$(\Ad_H\xi)_{(1)} = \xi_{(1)},$$ 
so $\Ad_H\xi - \xi_{(1)}'$ is a nilpotent derivation. Let 
$\xi_S:= \Ad_{H^{-1}} \xi_{(1)}'$ and  $\xi_N:= \Ad_{H^{-1}} (\Ad_H\xi -\xi_{(1)}')$.
Then $\xi_S$ is semisimple, $\xi_N$ is nilpotent and $\xi=\xi_S+\xi_N$. Moreover, 
\begin{eqnarray*}
\Ad_H[\xi_S,\xi_N] = [\xi_{(1)}',\Ad_H\xi -\xi_{(1)}']=[\xi_{(1)}', \Ad_H\xi]  = 0.
\end{eqnarray*}
Thus $[\xi_S,\xi_N]=0$ and this complete the proof.
\end{proof}

\section{Noncommutative Saito theorem}\label{sec:proof}

This section is devoted to establish a noncommutative analogue of the well-known Saito's theorem on  hypersurfaces of isolated  singularity. Throughout, let  $F$ be a fixed complete free algebra $k\lgg x_1,\ldots,x_n\rgg$ over a field $k$. We assume that $k$ is algebraically closed and of characteristic $0$, and we consider the rational number field $\mathbb{Q}$ as a subfield of $k$ in the natural way.

\begin{theorem}(NC Saito Theorem)\label{Saito}
Let $\Phi\in F_{\ol{\cy}}$ be a  superpotential of order $\geq3$ such that the Jacobi algebra associated to $\Phi$ is  finite dimensional. Then $\Phi$ is quasi-homogeneous if and only if  $\Phi$ is right equivalent to a weighted-homogenous superpotential of type $(r_1,\ldots,r_n)$ for some rational numbers $r_1,\ldots,r_n$ lie strictly between $0$ and $1/2$. Moreover, in this case, all such types $(r_1, \ldots, r_n)$ agree with each other up to permutations on the indexes $1,\ldots, n$. 
\end{theorem}

We address the proof of the above theorem after several lemmas.

\begin{lemma}\label{zero-criterion-develop}
Develop a formal series $f\in F$ in eigenvectors of a semisimple derivation $\xi\in\der_k^+(F)$ as $f=\sum_a f_a$. Then $f\in [F,F]^{cl}$ if and only if $f_a\in [F,F]^{cl}$ for each eigenvalue $a$ of $\xi$.
\end{lemma}
\begin{proof}
Since any automorphism of $F$ preserves $[F,F]^{cl}$, we may assume that  $\xi$ has $x_1,\ldots, x_n$ as eigenvectors. The result follows from the facts that the commutator of any two words is an eigenvector of $\xi$ and every formal series in $[F,F]^{cl}$ is a formal sum of such commutator. 
\end{proof}

\begin{lemma}\label{nilpotent-derivation}
Let $\Phi\in F_{\ol{\cy}}$ be a superpotential such that $\Phi_{\#}(\xi) = b\cdot \Phi$ for some scalar $b \in k$ and some nilpotent derivation $\xi\in \der_k^+(F)$. Then either $\Phi=0$ or $b=0$.
\end{lemma}
\begin{proof}
Suppose $\Phi\neq0$. Let $f$ be the canonical representative of $\Phi$. Develop $f$ as $f=\sum_{i\geq p} f_{(i)}$ with $f_{(i)}$ homogeneous of degree $i$ and $f_{(p)}\neq0$. Since $\xi$ is nilpotent, $\xi^r (\fm^p) \subseteq \fm^{p+1}$ for some $r\gg0$. So $\xi^r(f)$ has a decomposition $\xi^r(f)= \sum_{i\geq p+1} \xi^r(f)_{(i)}$ with $\xi^r(f)_{(i)}$ homogeneous of degree $i$. Then 
\[
b^r f_{(p)} + \sum_{i\geq p+1}\big (b^r f_{(i)} - \xi^r(f)_{(i)} \big) = b^r  f - \xi^r(f) \in [F,F]^{cl}.
\]
Consequently, $b^r f_{(p)}\in [F,F]^{cl}$. Since $f_{(p)}$ is in the canonical form, $b^rf_{(p)} =0$ and hence $b=0$. 
\end{proof}

\begin{lemma}\label{semisimple-preimage}
Let $\Phi\in F_{\ol{\cy}}$ be a superpotential such that  $\Phi_{\#}(\xi) = b \cdot \Phi$ for some scalar $b\in k$ and  some derivation $\xi\in \der_k^+(F)$.  Then $\Phi_{\#} (\xi_S) = b\cdot \Phi$ and $\Phi_{\#}(\xi_N)=0$, where $\xi_S$ and $\xi_N$ are the semisimple part  and the nilpotent part of $\xi$ respectively. 
\end{lemma}

\begin{proof}
Let $f$ be the canonical representative of $\Phi$.
Develop $f$ in eigenvectors of $\xi_S$ as $f=\sum_{a} f_a$. Since  $\pi(\xi(f) -bf)= \Phi_{\#}(\xi) -b\cdot \Phi =0$, where $\pi: F\to F_{\ol{\cy}}$ is the projection map, we have
$$
 \sum_a \xi_N(f_a)  + (a-b)f_a    \in  [F,F]^{cl}.
$$
Since $\xi_S(\xi_N(f_a)) = \xi_N(\xi_Sf_a) = a\xi_N(f_a)$, it follows that $(a-b)f_a +\xi_N(f_a)$ is an eigenvector of $\xi_S$ with eigenvalue $a$. Then Lemma \ref{zero-criterion-develop} tells us that 
\[
\pi(f_a)_{\#} (\xi_N) - (b-a) \cdot \pi (f_a) = \pi \big( \xi_N(f_a)+ (a-b) f_a \big ) =0
\] for every eigenvalue $a$ of $\xi_S$. So by Lemma \ref{nilpotent-derivation}, either $a=b$ or $f_a\in [F,F]^{cl}$ for every eigenvalue $a$ of $\xi_S$. Now we have two cases. If $b$ is not an eigenvalue of $\xi_S$ then $f\in [F,F]^{cl}$ and hence $\xi_S(f)- bf \in [F,F]^{cl}$; if $b$ is an eigenvalue of $\xi_S$ then $f- f_b =\sum_{a\neq b} f_a \in [F,F]^{cl}$ and hence 
\[\xi_S(f) -bf = \xi_S(f-f_b) -  b(f-f_b) \in [F,F]^{cl}.\] 
In both cases,  $\Phi_{\#} (\xi_S) - b\cdot \Phi = \pi(\xi_S(f) -bf) =0$.  Finally, $\Phi_{\#}(\xi_N) =\Phi_{\#}(\xi) - \Phi_{\#}(\xi_S) =0$.
\end{proof}

\begin{lemma}\label{word-constitution}
Let $\Phi\in F_{\ol{\cy}}$ be a  superpotential with finite dimensional  Jacobi algebra. Suppose 
\[
\Phi = \pi( g_{l+1}x_{l+1}) +\cdots + \pi(g_nx_n) + \pi(h),
\]
where $l<n$,  $g_{l+1},\ldots, g_n \in k\lgg x_1,\ldots, x_l\rgg$ and all monomials in $h \in F$  are of total degree $\geq 2$ in $x_{l+1},\ldots, x_n$. Then $l\leq n/2$ and there are at least $l$ nonzero formal series among $g_{l+1},\ldots, g_n$.
\end{lemma}

\begin{proof}
Let $k[[x_1,\ldots, x_l]]$ be the commutative algebra of power series in $l$  indeterminates. Let $\mathfrak{a}$ be the image of the Jacobi ideal ${\rm J}(F,\Phi)$ under the algebra homomorphism $\tau: F \to k[[x_1,\ldots, x_l]]$ given by $x_i\mapsto x_i$ for $i=1,\ldots, l$ and $x_i\mapsto 0$ for $i=l+1,\ldots, n$. Clearly, $\mathfrak{a}$ is a finite codimensional  proper ideal of  $k[[x_1,\ldots, x_l]]$ generated by $\tau(g_{l+1}),\ldots, \tau(g_n)$.   By the  well-known Krull's height theorem, $\mathfrak{a}$ has at least $l$ generators as a two-sided ideal of $k[[x_1,\ldots, x_l]]$, so there are at least $l$ nonzero power series among $\tau(g_{l+1}),\ldots, \tau(g_n)$. The result follows immediately.
\end{proof}

\begin{lemma}\label{weight-homo-eigenvector}
Let $\Phi\in F_{\ol{\cy}}$ be a superpotential of order $\geq 3$ such that the Jacobi algebra associated to $\Phi$ is finite dimensional. Suppose that   $\Phi_{\#}(\xi) = b \cdot \Phi$ for some nonzero $b\neq0$ and some semisimple  derivation $\xi\in \der_k^+(F)$ that has $x_1,\ldots, x_n$ as eigenvectors. 
Then $\Phi$ is weighted-homogeneous of type $(r_1,\ldots, r_n)$ for some rational numbers $r_1,\ldots, r_n$ lie strictly between $0$ and $1/2$. 
\end{lemma}

\begin{proof} 
By assumption, $\xi(x_i) =a_ix_i$ for $i=1,\ldots, n$, where $a_i\in k$. Let $c_1,\ldots, c_p$ be a basis of the vector space $\mathbb{Q}a_1+\cdots +\mathbb{Q}a_n+\mathbb{Q}b$ over $\mathbb{Q}$. Then
\[
(a_1,\ldots, a_n,b)^T = D\cdot (c_1,\ldots, c_p)^T
\]
for some matrix $D=(d_{ij})$ of type $(n+1)\times p$ with rational number entries. Since $b\neq 0$, the last row of $D$ is nonzero. Without lost of generality, we may assume $d_{n+1,1}\neq0$. Define
\[
(r_1,\ldots, r_n):=  (d_{1,1}/d_{n+1,1},\ldots, d_{n1}/d_{n+1,1}).
\]
Clearly, for any integers $m_1,\ldots, m_n$, if $m_1a_1+\cdots m_na_n=b$ then $(m_1,\ldots, m_n,-1) \cdot D =0$ and hence $m_1r_1+\cdots +m_nr_n=1$. Let $f$ be the canonical representative of $\Phi$. One has
$$\xi(f) =bf$$  
because $\xi(f)$ and $bf$ are both canonical representative of $b\cdot \Phi$.  It follows that for any word $w=x_{i_1}\cdots x_{i_s}$ that occurs in $f$, one has
$$
m_1a_1+\cdots +m_na_n = a_{i_1} +\cdots +a_{i_s} =b,
$$
 where $m_i$ is the occurrences of $x_i$ in the word $w$, and therefore
$$r_{i_1}+\cdots+r_{i_s}= m_1r_1+\cdots +m_nr_n =1.$$ 
It remains to show $0<r_1,\ldots, r_n<1/2$.

Now for any real number $\varepsilon\geq0$, let $P_\varepsilon$ (resp. $Q_\varepsilon$) be the number of  indexes $i$  among $1,\ldots, n$ such that $r_i\leq -\varepsilon$ (resp. $r_i\geq 1/2+\varepsilon$). We claim that  for every real number $\varepsilon \geq0$,
\[
P_\varepsilon \leq Q_{2\varepsilon+1/2} \quad \text{and} \quad Q_\varepsilon \leq P_{2\varepsilon}.
\]
To see the first inequality, we may assume $r_1,\ldots, r_{P_\varepsilon} \leq -\varepsilon$, up to permutation on indeterminates. Then $f$ contains no word constitutes with letters $x_1,\ldots, x_{P_\varepsilon}$.  By Lemma \ref{word-constitution} and the assumption that all terms of $f$ has degree $\geq3$, there are at least $P_\varepsilon$ indexes $i$ among $P_\varepsilon+1,\ldots, n$ such that $r_i\geq 1+2\varepsilon$, and so  $P_\varepsilon \leq Q_{2\varepsilon+1/2}$. The second inequality can be proved similarly.
 
From the above two inequalities, one has $P_\varepsilon \leq P_{4\varepsilon+1}$ and $Q_{\varepsilon} \leq Q_{2\varepsilon+1/2}$ for every real number $\varepsilon\geq0$. It follows that $P_0=Q_0=0$, or otherwise the finite set $\{r_1,\ldots, r_n\}$ is not bounded, which is absurd. Consequently, all rational numbers $r_1,\ldots, r_n$ lie strictly between $0$ and $1/2$. 
\end{proof}

\begin{proof}[\textbf{Proof of the equivalence statement of Theorem \ref{Saito}}] 
The backward implication is Lemma \ref{weighted-imply-quasi}. Next we proceed to show the forward implication. 

Assume that $\Phi$ is quasi-homogeneous. By Lemma \ref{quasi-derivation},  
$$\Phi_{\#}(\xi) = \Phi$$
for some  derivation $\xi\in \der_k^+(F)$. Then by Lemma \ref{semisimple-preimage}, 
\[
\Phi_{\#}(\xi_S) =\Phi.
\]
Choose an automorphism $H\in \Aut_k(F)$ such that the derivation $\Ad_H\xi_S=H\circ \xi_S \circ H^{-1}$ has $x_1,\ldots, x_n$ as eigenvectors. Note that 
$$
H(\Phi)_{\#} (\Ad_H\xi_S)= H(\Phi). 
$$
Then by Lemma \ref{weight-homo-eigenvector}, $H(\Phi)$ is weighted-homogeneous of type $(r_1,\ldots,r_n)$ for some rational numbers $r_1,\ldots,r_n$ lie strictly  between $0$ and $1/2$.  The result follows.
\end{proof}

To see the uniqueness statement of Theorem \ref{Saito}, we need the following lemma.
\begin{lemma}\label{semisimple-unique}
Let $\Phi\in F_{\ol{\cy}}$ be a  superpotential of order $\geq3$ such that the Jacobi algebra associated to $\Phi$ is  finite dimensional.  Given two semisimple derivations $\xi,\eta\in \der_k^{+}(F)$ that commute with each other, if $\Phi_{\#}(\xi) =\Phi_{\#}(\eta)$  then $\xi=\eta$.
\end{lemma}

\begin{proof}
By Proposition \ref{diagonalizable-simultaneous}, we may assume $\xi$ and $\eta$ both have $x_1.\ldots, x_n$ as eigenvectors with eigenvalue $r_1,\ldots, r_n$ and $s_1,\ldots, s_n$ respectively. 

Let $f$ be the canonical representative of $\Phi$. We claim that for each $1\leq i\leq n$, the formal series $f$ either has a monomial of the form $x_i^a$ for some $a\geq 3$ or has a monomial with exactly one occurrence of letters other than $x_i$.
Indeed, if the first case doesn't happen, then 
$$\Phi =\pi(f) = \sum_{p\neq i} \pi( g_p\cdot  x_p)  +  \pi(h), $$ with $g_p\in k\lgg x_i \rgg$ and with all monomials in $h$ has at least two occurrences in letters other than $x_i$.  By Lemma \ref{word-constitution}, there is at least one $p$ such that $g_p\neq0$, so the claim follows. 

Construct an $n\times n$ matrix $A=(a_{ij})$ with entries in $\mathbb{N}$ as follows. For each $1\leq i \leq n$, choose a monomial  in $f$ either of the form $x_i^a$ for some $a\geq 3$ or of the form $x_i^bx_px_i^c$ with $b+c\geq 2$ and $p\neq i$. Such a choose is assured by the above argument. Define the $i$-the row of $A$ to be $a e_i$ or $(b+c) e_i +e_p$, according to the choice of the monomial, where $e_i, e_p$ denote the canonical coordinate. Since $\xi(f) =\eta(f) $, it follows that
$$A\cdot (r_1,\ldots, r_n)^T =A \cdot (s_1,\ldots, s_n)^T.$$  Moreover, since 
$$a_{ii} > \sum_{p\neq i} a_{ip}, \quad i=1,\ldots, n ,$$  it follows that $A$ is an invertible matrix. Therefore $r_i=s_i$ for $i=1,\ldots,n$ and hence $\xi=\eta$.
\end{proof}

\begin{proof}[\textbf{Proof of the uniqueness statement of Theorem \ref{Saito}}]
Replacing $\Phi$ by an appropriate superpotential in its orbit,  we may assume $\Phi$ is itself weighted-homogeneous of type $r=(r_1,\ldots,r_n)$  with $r_1\leq \ldots \leq r_n$. Suppose that  
$H(\Phi)$ is weighted-homogeneous of type  $s=(s_1,\ldots, s_n)$  for some automorphisms $H$ of $F$. To see the result we must show that $r=s$ up to permutations.

 Let $\xi$ be the semisimple derivation of $F$ given by $\xi(x_i) =r_ix_i$, and let $\zeta := \Ad_{H^{-1}}\eta$, where $\eta$ is the semisimple derivation  given by $\eta(x_i) =s_ix_i$. 
Develop $\zeta(x_i)$ in eigenvectors of $\xi$ as $$\zeta(x_i)=\sum_a \zeta(x_i)_a,\quad i=1,\ldots, n.$$
Then define  for each eigenvalue $u$  of $\xi$ a derivation $\zeta_u \in \der_k^+(F)$ by
$\zeta_u(x_i) = \zeta(x_i)_{r_i+u}$.

Let $f$ be the canonical representative of $\Phi$.  Then $\xi(f) =f$ and
\[
f= \zeta(f) = \sum_u \zeta_u(f) \mod [F,F]^c
\]
where $u$ runs over all eigenvalues of $\xi$. It is easy to check that $\zeta_u(f)$ is an eigenvector of $\xi$ with eigenvalue $1+u$. Then by Lemma \ref{zero-criterion-develop},  one gets
\begin{eqnarray}\label{equation-zero}
\zeta_0(f)=f \mod [F,F]^{cl} 
\end{eqnarray}
and 
\begin{eqnarray}\label{equation-nonzero}
 \zeta_u(f) =0 \mod [F,F]^{cl}, \quad u\neq 0.
\end{eqnarray}
It is easy to check that $[\xi, \zeta_0]=0$. So $[\xi, (\zeta_0)_S]=0$ by Theorem \ref{decom-derivation}. One has $\Phi_{\#}(\zeta_0) =\Phi$ by Equation (\ref{equation-zero}), and hence $\Phi_{\#} ((\zeta_0)_S)=\Phi$ by  Lemma \ref{semisimple-preimage}. In addition,  $\Phi_{\#}(\xi) =\Phi$. Therefore,  
\[
\xi=(\zeta_0)_S
\]
by Lemma  \ref{semisimple-unique}. Thus the characteristic polynomial of the induced endomorphism of $\zeta_0$ on $\fm/\fm^2$  is
\[
(t-r_1)(t-r_2)\cdots (t-r_n).
\]
Note that the characteristic polynomial of the induced endomorphism of $\zeta$ on $\fm/\fm^2$, which equals to that of the induced endomorphism of  $\eta$ on $\fm/\fm^2$, is 
\[
(t-s_1)(t-s_2)\cdots (t-s_n).
\]
It remains to show that the induced linear endomorphisms  of $\zeta$ and $\zeta_0$ on $\fm/\fm^2$, denoted by $\widetilde{\zeta}$ and $\widetilde{\zeta_0}$ respectively, have the same characteristic polynomial.

We  first claim that the linear part of $\zeta(x_i)_{r_i+u} =\zeta_u(x_i)$  is zero  for $u<0$. Indeed, since $$\sum_{i=1}^n \zeta_u(x_i) \cdot D_{x_i} (f) = \zeta_u(f)\mod [F,F]^{cl},$$ it follows from equation  (\ref{equation-nonzero}) that 
\begin{eqnarray*}
\sum_{i=1}^n \zeta_u(x_i) \cdot D_{x_i} (f) =0 \mod [F,F]^{cl}, \quad u\neq 0.
\end{eqnarray*} 
Let  $\iota: F\to k[[x_1,\ldots, x_n]]$ be the algebra map given by $x_i\mapsto x_i$. Then $\iota(D_{x_1}(f)), \ldots, \iota(D_{x_n}(f))$ generates a finite codimensional ideal of $k[[x_1,\ldots, x_n]]$ and so they form a parameter system. By \cite[Theorem 8.21A (a,c)]{Hart}, any permutation of the sequence $\iota(D_{x_1}(f)), \ldots, \iota(D_{x_n}(f))$ is regular. Since
\[
\sum_{i=1}^n \iota(\zeta_u(x_i)) \cdot \iota(D_{x_i} (f) )=0, \quad u\neq 0,
\]
it follows that for each $1\leq i\leq n$ one has 
\[
\iota(\zeta_u(x_i)) \in \Big ( \iota(D_{x_1}(f)), \ldots, \wh{\iota(D_{x_i}(f))},\ldots, \iota(D_{x_n}(f)) \Big), \quad u\neq0.
\]
Since $D_{x_1}(f),\ldots, D_{x_n}(f)$ are all eigenvectors of $\xi$ of eigenvalue $\geq 1/2$ but $\zeta_u(x_i)$ is an eigenvector of $\xi$ of eigenvalue $r_i+u<1/2$ for $u<0$, it follows that 
$$\iota(\zeta_u(x_j)) =0, \quad u<0.$$
Since the linear part of $\zeta_u(x_i)$ coincide with the linear part of $\iota(\zeta_u(x_i))$, the claim follows.

Now note that $ r_{1} =\ldots =r_{l_1} < r_{l_1+1} = \ldots =r_{l_2} <\ldots < r_{l_{p-1}+1} =\ldots =r_n$ for
 some integers $0=l_0< \ldots< l_p=n$. By the above claim, for $l_q+1\leq i \leq l_{q+1}$ one has 
\begin{eqnarray*}
\widetilde{\zeta}(x_i) &=& \sum_{j=l_q+1}^{l_{q+1}} a_{ji}\cdot  x_j  + \sum_{j> l_{q+1}} a_{ji} \cdot x_j \\ 
\widetilde{\zeta_0}(x_i) &=& \sum_{j=l_q+1}^{l_{q+1}} a_{ji}\cdot  x_j.
\end{eqnarray*}
Compare the matrices of $\widetilde{\zeta}$ and $\widetilde{\zeta_0}$ with respect to the basis $x_1,\ldots, x_n$, one gets that  the characteristic polynomial of $\widetilde{\zeta}$ and $\widetilde{\zeta_0}$ are equal. This completes the proof.
\end{proof}

\begin{remark}
Let  $\iota: F\to k[[x_1,\ldots, x_n]]$ be the algebra homomorphism given by $x_i\mapsto x_i$ for $i=1,\ldots, n$.  It induces a map $\tilde{\iota}: F_{\ol{\cy}} \to k[[x_1,\ldots, x_n]]$. We call  $\tilde{\iota}(\Phi)$ the \emph{abelianization of $\Phi$} for any superpotential $\Phi\in F_{\ol{\cy}}$.  It is easy to check the following statements:
\begin{enumerate}
\item[(1)] The abelianization of right equivalent superpotentials are right equivalent as power series;
\item[(2)] The abelianization of a weighted-homogeneous superpotential is weighted-homogeneous of the same type as a power series;
\item[(3)] The abelianization of a quasi-homogeneous superpotential  is quasi-homogeneous as a power series. 
\end{enumerate}
Here the term ``right equivalence'' and ``weighted-homogeneous''  for power series are defined  in the obvious way, and a power series is called quasi-homogenous if it is contained in the ideal generated by its partial derivatives.
Note that these terminologies are not quite the same as that of \cite{Sait}. 

From the above three statements, the uniqueness statement of Theorem \ref{Saito} follows immediately from \cite[Lemma 4.3]{Sait}. However, we give a direct demonstration as above for completeness and reader's convenience.  Our argument is essential the same as that of Saito's, but with more details.  Of course,  some tricks are employed to deal with the noncommutativity.
In addition, our argument used Lemma \ref{semisimple-unique} (and hence Proposition \ref{diagonalizable-simultaneous}), which has an interest in its own right.
\end{remark}



\begin{thebibliography}{XXXX}
\bibitem{Am} C. Amiot, \emph{Cluster categories for algebras of global dimension 2 and quivers with potential.} (English, French summary)
Ann. Inst. Fourier (Grenoble) 59, no. 6 (2009): 2525--2590.
\bibitem{DWZ} H. Derksen, J. Weyman,  A. Zelevinsky, \emph{Quivers with potentials and their representations I: Mutations}, Selecta Math. 14 (2008): 59-119.
\bibitem{DW13} W. Donovan, M. Wemyss, \emph{Noncommutative deformations and flops}, Duke Math. J. 165, no. 8 (2016): 1397-1474.
\bibitem{Hart} R. Hartshorne, \emph{Algebraic geometry}, corrected 8th printing, Graduate Texts in Mathematics \textbf{52}, Springer (1997).
\bibitem{HZ}Z. Hua and G.-S. Zhou,  \emph{Noncommutative Mather-Yau theorem and its applications to Clabi-Yau algebras and birational geometry}, preprint, arXiv: 1803.06128v3 (2018).
\bibitem{Hump} J. E. Humphreys, \emph{ Introduction to Lie algebras and representation theory},  Graduate Texts in Mathematics \textbf{9}, Springer (1972).
\bibitem{MY} J. N. Mather, S. S.-T. Yau, \emph{Classification of isolated hypersurfaces singularitiesby by their moduli algebras}, Invent. Math. 69 (1982): 243-251.
\bibitem{RRS} G.-C. Rota, B. Sagan, P. R. Stein, \emph{A cyclic derivative in noncommutative algebra}, J. Alg. 64 (1980): 54-75.
\bibitem{Sait} K. Saito, \emph{Quasihomogene isolierte Singularit{\"a}ten von Hyperfl{\"a}chen},  Invent. Math. 14 (1971): 123-142.
\end{thebibliography}
\end{document}